\documentclass[11pt,a4paper]{article}
\usepackage[margin=25mm,top=24mm,bottom=25mm]{geometry}
\usepackage[T1]{fontenc}
\usepackage[utf8]{inputenc}
\usepackage{lmodern}
\usepackage{amsmath,amssymb,amsthm,mathtools,bm}
\usepackage{booktabs,array,longtable,enumitem,microtype}
\usepackage{tikz}
\usetikzlibrary{arrows.meta,positioning,calc}
\usepackage[unicode,colorlinks=true,linkcolor=black,citecolor=black,urlcolor=blue!45!black]{hyperref}
\usepackage{bookmark}
\hypersetup{pdftitle={Local Shellings of Broken-Circuit Complexes via Boundary States and the Chromatic Specialization of the Negami Polynomial},pdfauthor={IWAO MIZUKAI},pdfsubject={Boundary gluing, relative shellings, Orlik--Terao ideals, and state polytopes}}
\theoremstyle{plain}
\newtheorem{theorem}{Theorem}[section]
\newtheorem{lemma}[theorem]{Lemma}
\newtheorem{proposition}[theorem]{Proposition}
\newtheorem{corollary}[theorem]{Corollary}
\theoremstyle{definition}
\newtheorem{definition}[theorem]{Definition}
\newtheorem{example}[theorem]{Example}
\newtheorem{remark}[theorem]{Remark}
\newcommand{\BC}{\operatorname{BC}}
\newcommand{\NBC}{\operatorname{NBC}}
\newcommand{\Hilb}{\operatorname{Hilb}}
\newcommand{\conv}{\operatorname{conv}}
\newcommand{\lk}{\operatorname{lk}}
\newcommand{\inid}{\operatorname{in}}

\newcommand{\Bell}{\operatorname{Bell}}

\newcommand{\QQ}{\mathbb Q}
\newcommand{\RR}{\mathbb R}
\newcommand{\ZZ}{\mathbb Z}
\newcommand{\one}{\mathbf1}
\newcommand{\Sc}{\mathcal S}

\newcommand{\Gc}{\mathcal G}

\numberwithin{equation}{section}
\setlist{nosep,leftmargin=2em,topsep=5pt}
\allowdisplaybreaks[2]
\title{\bfseries Local Shellings of Broken-Circuit Complexes\\
via Boundary States and the Chromatic\\
Specialization of the Negami Polynomial}
\author{IWAO MIZUKAI\\[4pt]\small Chiba Keizai University Affiliated High School, Japan}
\date{}
\begin{document}
\maketitle
\begin{abstract}
For edge-disjoint graphs meeting along a finite boundary, we construct relative shellings of NBC faces with prescribed boundary connectivity. The construction adds paths in spanning forests to the restriction faces of auxiliary rooted-forest complexes. When every edge of the first graph precedes every edge of the second, these relative shellings combine into a shelling with consecutive connectivity blocks. The resulting nonnegative decomposition of the $h$-polynomial realizes the chromatic specialization of the Negami splitting identity. We also give path formulas for gluing graphic Orlik--Terao Gr\"obner bases across boundaries of size two or three, and a virtual-clique elimination formula for arbitrary boundaries. An exponential output lower bound at treewidth two distinguishes explicit basis enumeration from coefficient computation at fixed treewidth. Finally, state polytopes and their Cayley lifts describe an obstruction to linear block orders in the original coordinates and a lifted realization of these orders.
\end{abstract}

\section{Introduction and context}
The Negami polynomial is a graph invariant equipped with an edge-subset expansion and vertex-splitting identities \cite{Negami,NegamiKawagoe}. Its chromatic specialization is related to Whitney's broken-circuit expansion and determines the face numbers and the $h$-polynomial of a broken-circuit complex. The purpose of this paper is to interpret the individual terms of the splitting identity through restriction faces of local relative shellings.

The shellability of broken-circuit complexes and their lexicographic shellings are classical \cite{Bjorner,Ziegler}. The theorem that circuit polynomials form a universal Gr\"obner basis, with degeneration to the broken-circuit ring, is due to Proudfoot and Speyer \cite{ProudfootSpeyer}. We use these results as foundational input. Modular decompositions of Orlik--Terao algebras \cite{Denham}, the Cayley construction \cite{Huber}, and polynomial evaluation on graphs of bounded treewidth \cite{Noble} also belong to established theory. Our aim is to give a unified, explicit formulation in terms of boundary states, with the necessary hypotheses and proofs, rather than to claim that every ingredient is independently new.

We distinguish three questions:
\begin{enumerate}
\item Can local information determine which faces are NBC for a prescribed edge order?
\item Can facets be placed in consecutive state blocks so that their new faces form shelling intervals?
\item Can polynomial coefficients be computed without enumerating all facets or an entire Gr\"obner basis?
\end{enumerate}
These are different assertions. In particular, equality of initial ideals or Hilbert series alone does not imply a shelling with prescribed state blocks.

Our principal results are the relative shelling for a fixed state (Theorem~\ref{thm:relative}), a one-step block construction for an arbitrary boundary (Theorem~\ref{thm:block}), its relation to the Negami splitting identity (Theorem~\ref{thm:negami-h}), elimination using a virtual complete graph (Theorem~\ref{thm:clique}), and an exponential output example at treewidth two (Theorem~\ref{thm:exponential}). The block construction requires an edge-order hypothesis. Although we state and prove the three-variable splitting identity, the nonnegative interpretation by restriction faces concerns its chromatic specialization.

\section{Ordered graphs, broken-circuit complexes, and relative \texorpdfstring{$h$}{h}-polynomials}
All graphs are finite, and edges retain their labels. The original graphs are loopless; quotient graphs may have loops and parallel edges. Fix a total order $<$ on the edge set. A pair of parallel edges is regarded as a circuit.

\begin{definition}
The \emph{broken circuit} associated with a circuit $C$ is $C\setminus\{\max C\}$. An edge set containing no broken circuit is an \emph{NBC face}; the collection of all such faces is denoted by $\BC(G)$. If $G$ has a loop, the empty set is a broken circuit, and $\BC(G)$ is the void complex, with no faces.
\end{definition}
The convention of deleting the smallest edge, often used in the literature, is obtained by reversing the edge order. We delete the largest edge throughout. Every NBC face is a forest. For a connected loopless graph, $\BC(G)$ is pure of rank $r=|V(G)|-1$, and its facets are the NBC spanning trees.

\begin{lemma}[Sequential connectivity criterion]\label{lem:criterion}
Write $A_{<e}=\{a\in A:a<e\}$. Then
\begin{equation}\label{eq:criterion}
A\in\BC(G)\quad\Longleftrightarrow\quad
\text{for every edge $e=xy$, $x$ and $y$ are disconnected in $(V(G),A_{<e})$.}
\end{equation}
\end{lemma}
\begin{proof}
If a path $P\subseteq A_{<e}$ joins $x$ to $y$, then $e$ is the largest edge of $P\cup\{e\}$, so $P$ is a broken circuit. Conversely, a broken circuit is a path joining the endpoints of the deleted largest edge using only smaller edges. A loop gives the same criterion with a path of length zero.
\end{proof}

For a pure complex $\Delta$ whose facets have cardinality $r$, set
\begin{equation}\label{eq:h}
h_\Delta(z)=\sum_{A\in\Delta}z^{|A|}(1-z)^{r-|A|}.
\end{equation}
For a relative complex $\Delta\setminus\Gamma$, take the sum over its faces using the same rank $r$. A \emph{shelling} is an ordering $F_1,\ldots,F_s$ of the facets such that, at each step, the new faces form a single Boolean interval:
\[
\{A\subseteq F_j:A\text{ has not appeared earlier}\}=[R(F_j),F_j].
\]
For a relative complex, all faces of $\Gamma$ are considered to have appeared before the first step. The face $R(F_j)$ is the \emph{restriction face}. Then
\begin{equation}\label{eq:h-shell}
h_\Delta(z)=\sum_j z^{|R(F_j)|},
\end{equation}
with the analogous formula for a relative complex: the contribution of each interval to \eqref{eq:h} is $z^{|R(F_j)|}$.

For a connected graph, write $H_G=h_{\BC(G)}$. Set $H_G=0$ for a quotient graph with a loop, and $H_G=1$ for the graph consisting of a single vertex and no edges.

\begin{proposition}[Classical shellability theorem]\label{prop:known-shell}
The broken-circuit complex of a loopless matroid is shellable. Under our convention, the lexicographic facet order that places first the facet containing the largest element of the symmetric difference is a shelling. Links of faces in a pure shellable complex are also shellable.
\end{proposition}
The first assertion follows from \cite{Bjorner,Ziegler}. For the second, retain, in their original order, the facets containing the specified face and delete that face from each of them. Restricting the original shelling intervals verifies the assertion.

\section{Relative shellings with a fixed boundary state}
Throughout this section, $G$ is connected and loopless, and $\varnothing\ne B\subseteq V(G)$ is a boundary. Let $\Pi(B)$ denote the set of partitions of $B$, ordered by refinement: $\pi\le\sigma$ means that $\pi$ refines $\sigma$. Write $\pi(A)$ for the connectivity partition induced on $B$ by an edge set $A$. In general,
\[
\Sc_\pi=\{A\in\BC(G):\pi(A)=\pi\}
\]
is not a subcomplex. Removing an edge from a path between boundary vertices can change the state.

Choose one representative $s_C$ from each block $C\in\pi$, and put $S=\{s_C:C\in\pi\}$. The collection
\begin{equation}\label{eq:DS}
D_S=\{A\in\BC(G):\text{distinct vertices of $S$ lie in distinct components of $A$}\}
\end{equation}
is a subcomplex.

\begin{lemma}[Link description using a virtual tree]\label{lem:link}
Add a tree $Q$ on $S$, placing all its edges after the original edges. Take $Q=\varnothing$ when $|S|=1$. Then
\begin{equation}\label{eq:link}
D_S=\lk_{\BC(G+Q)}Q.
\end{equation}
Consequently, $D_S$ is pure and shellable of rank $|V(G)|-|S|$. Its facets are NBC spanning forests with $|S|$ components, each containing exactly one vertex of $S$.
\end{lemma}
\begin{proof}
If $A\in D_S$, then $A\cup Q$ is a forest. The edges of $Q$ do not enter the criterion \eqref{eq:criterion} for an original edge, since they are larger. The criterion also holds for the edges of $Q$, because $A\cup Q$ is a forest. Hence $A\cup Q\in\BC(G+Q)$.

Conversely, if $A\cup Q$ is NBC, then so is $A$. If a component of $A$ contained two representatives, adding the edges of the tree $Q$ would create a cycle. Thus $A$ separates $S$, proving \eqref{eq:link}. Since $G+Q$ is connected and $|Q|=|S|-1$, the link has rank $|V(G)|-|S|$. Its facets have as many components as representatives, which proves the final assertion.
\end{proof}

\begin{theorem}[Relative shelling for a boundary state]\label{thm:relative}
Let $n=|V(G)|$ and $r_\pi=n-|\pi|$. If $\Sc_\pi$ is nonempty, it is a pure relative complex of rank $r_\pi$ and admits a relative shelling. More precisely, choose any shelling of $D_S$ and retain only its facets $F$ with state $\pi$. Let $P_\pi(F)$ be the union, over all $C\in\pi$ and $b\in C$, of the unique paths in $F$ from $s_C$ to $b$. The restriction face in the relative shelling is
\begin{equation}\label{eq:relative-R}
R_\pi(F)=R_D(F)\cup P_\pi(F).
\end{equation}
In particular,
\begin{equation}\label{eq:J-state}
J_{G,\pi}(z):=\sum_{A\in\Sc_\pi}z^{|A|}(1-z)^{r_\pi-|A|}
=\sum_F z^{|R_D(F)\cup P_\pi(F)|}\in\ZZ_{\ge0}[z].
\end{equation}
For an empty state, set $J_{G,\pi}=0$. No block hypothesis on the edge order is required.
\end{theorem}
\begin{proof}
The set $\Sc_\pi$ is upward closed in $D_S$. Indeed, adding edges to a face of state $\pi$ can merge two boundary blocks only by connecting their representatives, which would leave $D_S$. Therefore $D_S\setminus\Sc_\pi$ is a subcomplex. Extending a face of $\Sc_\pi$ to a facet of $D_S$ preserves its state, proving relative purity.

A subset $A$ of a forest $F$ of state $\pi$ has the same state if and only if $P_\pi(F)\subseteq A$. Hence each original shelling interval satisfies
\[
[R_D(F),F]\cap\Sc_\pi=[R_D(F)\cup P_\pi(F),F].
\]
A facet of $D_S$ of another state contains no face in $\Sc_\pi$. Deleting those facets therefore does not affect which relative faces have appeared earlier. This proves the relative shelling and \eqref{eq:J-state}.
\end{proof}

\section{One-step block composition along an arbitrary boundary}
Let $G_1,G_2$ be connected loopless graphs such that
\[
V(G_1)\cap V(G_2)=B,\qquad E(G_1)\cap E(G_2)=\varnothing.
\]
Put $k=|B|\ge1$, $n_i=|V(G_i)|$, $G=G_1\cup G_2$, and $r=n_1+n_2-k-1$. In this section, assume
\begin{equation}\tag{A}\label{eq:A}
e<f\qquad(e\in E(G_1),\ f\in E(G_2)).
\end{equation}
The quotient $G_2/\pi$ identifies all vertices within each block of $\pi$, retaining edge labels, loops, and parallel edges. Write $K_\pi=\BC(G_2/\pi)$.

\begin{lemma}[NBC criterion in a quotient graph]\label{lem:quotient}
For $A_i\subseteq E(G_i)$,
\begin{equation}\label{eq:quotient}
A_1\cup A_2\in\BC(G)
\iff A_1\in\BC(G_1)\ \text{and}\ A_2\in K_{\pi(A_1)}.
\end{equation}
Moreover, $\rho\le\pi$ implies $K_\pi\subseteq K_\rho$.
\end{lemma}
\begin{proof}
For an edge of $G_1$, no edge of $G_2$ enters \eqref{eq:criterion}. For an edge $e$ of $G_2$, all of $A_1$ has already appeared. Two vertices of $G_2$ are connected in $A_1\cup(A_2)_{<e}$ precisely when they are connected after identifying the boundary blocks of $\pi(A_1)$. This proves \eqref{eq:quotient}. Coarser identifications can only increase connectivity, so the same criterion proves the inclusion.
\end{proof}

\begin{theorem}[Composition of block shellings]\label{thm:block}
Under \eqref{eq:A}, there is a disjoint decomposition of relative face sets
\begin{equation}\label{eq:state-decomposition}
\BC(G)=\bigsqcup_{\pi\in\Pi(B)}\Sc_\pi*K_\pi.
\end{equation}
Order the states by any linear extension of refinement, with finer partitions first. Within each state, use the relative shelling from Theorem~\ref{thm:relative} for the outer order on $F$, and a shelling of $K_\pi$ for the inner order on $U$. Then the resulting order on $F\cup U$ is a shelling of $\BC(G)$, with restriction face
\begin{equation}\label{eq:global-R}
R(F\cup U)=R_D(F)\cup P_\pi(F)\cup R_{2,\pi}(U).
\end{equation}
Consequently,
\begin{equation}\label{eq:nonnegative}
H_G(z)=\sum_{\pi\in\Pi(B)}J_{G_1,\pi}(z)H_{G_2/\pi}(z)
\end{equation}
is a decomposition into polynomials with nonnegative coefficients.
\end{theorem}
\begin{proof}
Lemma~\ref{lem:quotient} gives the decomposition of face sets. The cardinalities of $F$ and $U$ are respectively
\[
n_1-|\pi|\quad\text{and}\quad n_2-k+|\pi|-1,
\]
whose sum is $r$. Every face in a state extends to such a union. Thus these unions are exactly the facets of the whole complex, each appearing once.

Consider a subface $A_1\cup A_2$ of $F\cup U$. Its left state $\rho=\pi(A_1)$ refines $\pi$. If $\rho<\pi$, then $A_2\in K_\pi\subseteq K_\rho$, so this face has already appeared in the earlier block for $\rho$. A face with left state $\pi$ cannot have appeared in an earlier block of a different state $\sigma$: containment in a facet of that block would require $\pi\le\sigma$, contrary to the chosen linear extension.

It therefore remains to test newness within state $\pi$. The outer relative shelling requires $R_D(F)\cup P_\pi(F)\subseteq A_1$, and the inner shelling requires $R_{2,\pi}(U)\subseteq A_2$. The new faces are exactly
\[
[R_D(F)\cup P_\pi(F)\cup R_{2,\pi}(U),\ F\cup U].
\]
This proves the shelling. The edge sets of the two sides are disjoint, so restriction-face cardinalities add, giving \eqref{eq:nonnegative}.
\end{proof}

\begin{remark}\label{rem:one-step}
Theorem~\ref{thm:block} allows a boundary of arbitrary size, but concerns a single gluing step. It does not provide a closed recursion that updates the required shellings of $D_S$ from previous records when the boundary changes. Reusing shellings of the auxiliary complexes $D_S$ is also different from unconditionally reusing arbitrary given shellings of $\BC(G_i)$.
\end{remark}

\section{The Negami splitting identity and its chromatic specialization}
\subsection{The three-variable state sum and the partition matrix}
We use the normalization
\begin{equation}\label{eq:negami}
f(G;t,x,y)=\sum_{A\subseteq E(G)}t^{\kappa_G(A)}x^{|A|}y^{|E(G)|-|A|}
\end{equation}
of the Negami polynomial \cite{Negami,NegamiKawagoe}. Here $\kappa_G(A)$ is the number of components of the spanning subgraph, including isolated vertices. In particular, the chromatic polynomial is
\begin{equation}\label{eq:chromatic}
P_G(q)=f(G;q,-1,1).
\end{equation}
Define the boundary state sums by
\begin{equation}\label{eq:state-negami}
N_{i,\pi}(t,x,y)=\sum_{\pi(A)=\pi}
t^{\kappa_{G_i}(A)-|\pi|}x^{|A|}y^{|E(G_i)|-|A|}.
\end{equation}
The exponent of $t$ is nonnegative because exactly $|\pi|$ components meet the boundary.

\begin{proposition}[Splitting by boundary states]\label{prop:split}
Without any assumption on the edge order,
\begin{align}
f(G;t,x,y)
&=\sum_{\pi,\rho\in\Pi(B)}t^{|\pi\vee\rho|}N_{1,\pi}N_{2,\rho},\label{eq:split-two}\\
&=\sum_{\pi\in\Pi(B)}N_{1,\pi}(t,x,y)f(G_2/\pi;t,x,y).\label{eq:split-one}
\end{align}
Here $\pi\vee\rho$ denotes the least common coarsening.
\end{proposition}
\begin{proof}
For edge sets $A_1,A_2$ of states $\pi,\rho$,
\[
\kappa_G(A_1\cup A_2)=\kappa_{G_1}(A_1)-|\pi|
+\kappa_{G_2}(A_2)-|\rho|+|\pi\vee\rho|.
\]
Components avoiding the boundary remain on their respective sides, while components meeting it are described by the common coarsening. The exponents of $x$ and $y$ add because the edge sets are disjoint. This proves \eqref{eq:split-two}; summing first over $A_2$ gives \eqref{eq:split-one}.
\end{proof}

Let $M_B(t)=(t^{|\pi\vee\rho|})_{\pi,\rho\in\Pi(B)}$. If $\mathbf F_i$ is the vector of quotient polynomials $f(G_i/\pi)$ and $\mathbf N_i$ is the vector of state sums, then
\[
\mathbf F_i=M_B(t)\mathbf N_i.
\]
Set $Z_{\pi,\tau}=\one_{\pi\le\tau}$ and $(t)_j=t(t-1)\cdots(t-j+1)$. Then
\begin{equation}\label{eq:partition-matrix}
M_B(t)=Z\operatorname{diag}((t)_{|\tau|})Z^{\mathsf T},\qquad
\det M_B(t)=\prod_{\tau\in\Pi(B)}(t)_{|\tau|}.
\end{equation}
Indeed, $t^{|\rho|}=\sum_{\tau\ge\rho}(t)_{|\tau|}$ classifies color assignments to boundary blocks according to their equal-color classes. Formally, this gives
\begin{equation}\label{eq:matrix-split}
f(G)=\mathbf F_1^{\mathsf T}M_B(t)^{-1}\mathbf F_2.
\end{equation}
The singularity of this matrix at particular values of $t$ is related to the issues studied in \cite{Burgos}. The denominator-free identities \eqref{eq:split-two} and \eqref{eq:split-one}, however, remain valid at these values.

\subsection{A sign-reversing involution preserving boundary states}
\begin{lemma}\label{lem:cancel}
For any edge order,
\begin{equation}\label{eq:N-nbc}
N_{G,\pi}(q,-1,1)=\sum_{A\in\Sc_\pi}(-1)^{|A|}q^{n-|A|-|\pi|}.
\end{equation}
\end{lemma}
\begin{proof}
For a non-NBC set $A$, choose the smallest edge $e$ whose endpoints are connected in $A_{<e}$, and pair $A$ with $A\triangle\{e\}$. Edges smaller than $e$ do not change, so the smallest violating edge is unchanged; this is an involution. The endpoints of $e$ are already connected by other edges, hence all components and the boundary state are preserved. The parity of the number of edges changes, so the contributions cancel. The remaining NBC faces are forests and have $n-|A|$ components.
\end{proof}

\begin{theorem}[Restriction-face enumerators and chromatic specialization]\label{thm:negami-h}
For $q=(z-1)/z$,
\begin{align}
J_{G,\pi}(z)&=(-z)^{n-|\pi|}N_{G,\pi}(q,-1,1),\label{eq:J-N}\\
H_G(z)&=(-z)^{n-1}\frac{P_G(q)}q.\label{eq:H-P}
\end{align}
Thus \eqref{eq:nonnegative} holds as a polynomial identity for state polynomials computed with arbitrary local edge orders. Under \eqref{eq:A}, each summand enumerates the restriction faces in an actual block of the shelling in Theorem~\ref{thm:block}.
\end{theorem}
\begin{proof}
Multiplying each summand of \eqref{eq:N-nbc} by $(-z)^{n-|\pi|}$ and substituting for $q$ gives $z^{|A|}(1-z)^{n-|\pi|-|A|}$. The same cancellation without fixing a boundary state gives \eqref{eq:H-P}. Apply the chromatic specialization to \eqref{eq:split-one} and use
\[
n-1=(n_1-|\pi|)+(n_2-k+|\pi|-1)
\]
to obtain \eqref{eq:nonnegative}. The state sums do not use an edge order, so neither do the polynomials $J_{G,\pi}$. Formula~\eqref{eq:global-R} gives the interpretation by an actual shelling.
\end{proof}

\begin{remark}
Proving the three-variable identity \eqref{eq:split-one} is different from interpreting all three variables by nonnegative weights on restriction faces. The shelling interpretation established here concerns $f(G;q,-1,1)$ and its quotient-graph counterparts.
\end{remark}

\section{Explicit formulas for two and three terminals}
\subsection{A two-stage decomposition for two terminals}
Let the boundary be $\{u,v\}$, put $\Delta=\BC(G)$, and set
\[
D=\{A\in\Delta:u\not\leftrightarrow v\text{ in }A\},\qquad
K=\BC(G/uv),\qquad \widehat H_G=h_K.
\]
Denote the relative polynomial of the connected state by $J_{G;u,v}$.
\begin{proposition}\label{prop:two}
We have
\begin{align}
h_D&=H_G+z\widehat H_G,\label{eq:D-two}\\
J_{G;u,v}&=H_G-(1-z)h_D
=z\{H_G-(1-z)\widehat H_G\}.\label{eq:J-two}
\end{align}
Moreover, under \eqref{eq:A},
\begin{equation}\label{eq:two-complex}
\begin{aligned}
\BC(G_1\cup G_2)&=(D_1*\Delta_2)\cup(\Delta_1*K_2),\\
(D_1*\Delta_2)\cap(\Delta_1*K_2)&=D_1*K_2,
\end{aligned}
\end{equation}
and
\begin{equation}\label{eq:two-positive}
H_{G_1\cup G_2}=h_{D_1}H_{G_2}+J_{G_1;u,v}\widehat H_{G_2}.
\end{equation}
\end{proposition}
\begin{proof}
Writing the two state sums as $N_0,N_1$, we have $P_G=q^2N_0+qN_1$ and $P_{G/uv}=q(N_0+N_1)$. Substitution of \eqref{eq:J-N} gives \eqref{eq:D-two} and \eqref{eq:J-two}. The inclusion $K_2\subseteq\Delta_2$ and \eqref{eq:state-decomposition} give \eqref{eq:two-complex}. The new faces in the second stage form $(\Delta_1\setminus D_1)*K_2$, and their restriction faces are
\[
R_{\Delta_1}(T)\cup P_T(u,v)\cup R_{K_2}(F).
\]
These account for the second summand of \eqref{eq:two-positive}.
\end{proof}

Substituting \eqref{eq:D-two} and \eqref{eq:J-two} into \eqref{eq:two-positive} gives the symmetric, order-independent identity
\begin{equation}\label{eq:parallelH}
H_{G_1\cup G_2}=H_1H_2+z(H_1\widehat H_2+\widehat H_1H_2)
-z(1-z)\widehat H_1\widehat H_2.
\end{equation}
Although this expression contains a subtraction, \eqref{eq:two-positive} is a nonnegative decomposition arising from a relative shelling.

\begin{remark}[Relation to matroid $2$-sums]
Add a virtual edge $p=uv$ to each $G_i$. Under the usual conditions for a $2$-sum, namely that $p$ is neither a loop nor a coloop,
\[
M(G_1\cup G_2)=M(G_1+p)\oplus_p M(G_2+p).
\]
The description of circuits as local circuits or unions of a $u$--$v$ path from each side agrees with the circuit description of this $2$-sum. The decomposition \eqref{eq:two-complex} for a prescribed edge order and the correction to restriction faces nevertheless require the ordered argument above.
\end{remark}

\subsection{The five states for three terminals}
For $B=\{u,v,w\}$, the partitions are
\[
\pi_0=u|v|w,\quad \pi_{uv}=uv|w,\quad
\pi_{uw}=uw|v,\quad\pi_{vw}=vw|u,\quad\pi_t=uvw.
\]
Table~\ref{tab:states} gives one choice of representatives and correction paths.
\begin{table}[htbp]\centering\small
\caption{Auxiliary complexes and restriction-face corrections for the five states.}\label{tab:states}
\begin{tabular}{lll}\toprule
State $\pi$ & Auxiliary complex & $P_\pi(F)$\\\midrule
$u|v|w$ & $D_{\{u,v,w\}}$ & $\varnothing$\\
$uv|w$ & $D_{\{u,w\}}$ & $P_F(u,v)$\\
$uw|v$ & $D_{\{u,v\}}$ & $P_F(u,w)$\\
$vw|u$ & $D_{\{u,v\}}$ & $P_F(v,w)$\\
$uvw$ & $\BC(G_1)$ & $P_F(u,v)\cup P_F(u,w)$\\\bottomrule
\end{tabular}
\end{table}
Under \eqref{eq:A}, the blocks can be attached in the order $\pi_0$, the three two-block states in any order, and $\pi_t$. The splitting matrix is
\begin{equation}\label{eq:M3}
M_3(q)=
\begin{pmatrix}
q^3&q^2&q^2&q^2&q\\
q^2&q^2&q&q&q\\
q^2&q&q^2&q&q\\
q^2&q&q&q^2&q\\
q&q&q&q&q
\end{pmatrix},\qquad \det M_3=q^5(q-1)^4(q-2).
\end{equation}

\begin{corollary}[Recovery from five quotient polynomials]\label{cor:five}
Let $H=H_{G_1}$, $H_{ab}=H_{G_1/(a=b)}$, and $H_t=H_{G_1/(u=v=w)}$. Abbreviating the state indices by $0,ab,t$, we have
\begin{align}
J_0&=\frac{H+z(H_{uv}+H_{uw}+H_{vw})+2z^2H_t}{1+z},\label{eq:five0}\\
J_{ab}&=z\{J_0-H_{ab}-zH_t\},\label{eq:fivepair}\\
J_t&=H-(1-z)^2J_0-(1-z)(J_{uv}+J_{uw}+J_{vw}).\label{eq:fivet}
\end{align}
All these expressions are polynomials with nonnegative coefficients.
\end{corollary}
\begin{proof}
Apply \eqref{eq:J-N} to $\mathbf F=M_3\mathbf N$. The resulting relations are
\[
H_t=J_0-z^{-1}(J_{uv}+J_{uw}+J_{vw})+z^{-2}J_t,
\]
\[
H_{ab}=(1-z)J_0-(1-z)z^{-1}J_{ab}
+\sum_{cd\ne ab}J_{cd}-z^{-1}J_t,
\]
together with \eqref{eq:fivet}, rearranged. Solving gives \eqref{eq:five0}--\eqref{eq:fivet}. Polynomiality and nonnegativity follow from Theorem~\ref{thm:relative}. In particular, the numerator in \eqref{eq:five0} is divisible by $1+z$.
\end{proof}

\begin{example}[A five-state computation]\label{ex:five-numeric}
Let $G_1=K_{3,2}$, with its part of size three as boundary $\{u,v,w\}$, and let $G_2$ be a three-edge star with the same boundary as its leaves. Each quotient $G_2/\pi$ becomes a tree after parallel edges are simplified, so $H_{G_2/\pi}=1$. For $G_1$,
\[
H=1+2z+3z^2+z^3,\qquad
H_{uv}=H_{uw}=H_{vw}=1+z+z^2,\qquad H_t=1.
\]
Corollary~\ref{cor:five} gives
\[
J_0=1+4z+4z^2,\qquad
J_{uv}=J_{uw}=J_{vw}=2z^2+3z^3,\qquad
J_t=2z^3+5z^4.
\]
Thus the glued graph $K_{3,3}$ has
\[
H_{K_{3,3}}=1+4z+10z^2+11z^3+5z^4.
\]
At $z=1$, the numbers of facets in the five states are $9,5,5,5,7$, giving $31$ in total. The same gluing will provide the geometric obstruction in Proposition~\ref{prop:midpoint}.
\end{example}

\subsection{What fails without the edge-order hypothesis}
Split a four-cycle into two paths of length two, $G_1=\{e_1,e_4\}$ and $G_2=\{e_2,e_3\}$, and order the edges by $e_1<e_2<e_3<e_4$. Its unique broken circuit is $\{e_1,e_2,e_3\}$. Yet $\{e_1\}\in D_1$ and $\{e_2,e_3\}\in\Delta_2$, so $D_1*\Delta_2$ contains this forbidden face. Therefore \eqref{eq:two-complex} fails for this order.

Moreover, both $\{e_1\}$ and $\{e_4\}$ have the local state separating the terminals. Their unions with $\{e_2,e_3\}$ are, respectively, non-NBC and NBC. Connectivity partitions alone therefore do not suffice for recognition under an interleaved edge order. This example does not prove the nonexistence of all possible block shellings.

\section{Low-degree coefficients and series--parallel computation}
\subsection{Distance, shortest paths, and the next coefficients}
In this subsection, $G$ is connected and simple. Let $d=d_G(u,v)$ and let $a$ be the number of shortest $u$--$v$ paths.
\begin{proposition}\label{prop:distance}
\begin{equation}\label{eq:distance}
J_{G;u,v}(z)=az^d+O(z^{d+1}).
\end{equation}
\end{proposition}
\begin{proof}
In the face sum \eqref{eq:J-state}, every face connecting $u$ and $v$ has at least $d$ edges. Such a face with exactly $d$ edges is a shortest path. A shortest path is induced and contains no broken circuit. Hence the first coefficient counts shortest paths.
\end{proof}

Higher coefficients admit expressions in terms of small induced connected subgraphs containing $u,v$. For $W\subseteq V(G)$, let $b(W)$ be the number of NBC spanning trees of $G[W]$, with $b(W)=0$ when $G[W]$ is disconnected. For $U=V(G)\setminus W$, define
\[
\Phi_{G[U]}(z)=\sum_{A\in\BC(G[U])}z^{|A|}(1-z)^{|U|-|A|}.
\]
The reference exponent is $|U|$, rather than the usual rank in an $h$-polynomial. For the empty graph, set $\Phi=1$.

\begin{theorem}[Expansion by the component containing the terminals]\label{thm:component-J}
We have
\begin{equation}\label{eq:component-J}
J_{G;u,v}(z)=\sum_{W\supseteq\{u,v\}}
b(W)z^{|W|-1}\Phi_{G[V\setminus W]}(z).
\end{equation}
Put $\chi(U)=|U|-|E(G[U])|$, and let $\tau(U)$ be the number of triangles in $G[U]$. Let $\mathcal W_j$ consist of vertex sets $W$ containing $u,v$, of cardinality $d+1+j$, for which $G[W]$ is connected. Then
\begin{align}
[z^{d+1}]J
&=\sum_{W\in\mathcal W_1}b(W)
-\sum_{W\in\mathcal W_0}\chi(V\setminus W),\label{eq:Jnext}\\
[z^{d+2}]J
&=\sum_{W\in\mathcal W_2}b(W)
-\sum_{W\in\mathcal W_1}b(W)\chi(V\setminus W)\notag\\
&\quad+\sum_{W\in\mathcal W_0}
\left\{\binom{\chi(V\setminus W)}2-\tau(V\setminus W)\right\}.
\label{eq:Jnext2}
\end{align}
\end{theorem}
\begin{proof}
For an NBC face $A$ connecting $u,v$, let $W$ be the vertex set of their component. This component is an NBC spanning tree of $G[W]$, and the remainder is an arbitrary NBC face of $G[V\setminus W]$. Conversely, the union of two such faces is NBC in $G$ by Lemma~\ref{lem:criterion}: endpoints of edges between the two parts are disconnected in the selected face. The face weights now give \eqref{eq:component-J}.

If $G[U]$ has $m_U$ edges, its numbers of one-edge and two-edge NBC faces are $m_U$ and $\binom{m_U}2-\tau(U)$. Consequently,
\[
\Phi_{G[U]}(z)=1-\chi(U)z+
\left\{\binom{\chi(U)}2-\tau(U)\right\}z^2+O(z^3).
\]
If $|W|=d+1$ and $b(W)\ne0$, then $G[W]$ is itself a shortest path, so $b(W)=1$. Comparing coefficients proves the formulas. For negative integers $c$, the notation $\binom c2$ means $c(c-1)/2$.
\end{proof}

\begin{example}
For the four-cycle,
\[
J_{C_4;u,v}(z)=
\begin{cases}
z+z^2+z^3,&\text{$u,v$ are adjacent},\\
2z^2+z^3,&\text{$u,v$ are opposite}.
\end{cases}
\]
The least degree and its coefficient recover the distance and the number of shortest paths. Formulas~\eqref{eq:Jnext} and \eqref{eq:Jnext2} are signed structural expressions for higher coefficients; they do not assert termwise nonnegative counting. Nonnegativity of the resulting coefficients follows from the relative shelling.
\end{example}

\subsection{Numerical recursion for two-terminal series--parallel graphs}
Write $S$ and $P$ for series and parallel composition of two-terminal networks.
\begin{proposition}\label{prop:SP}
Set $\Psi=H_1H_2+z(H_1\widehat H_2+\widehat H_1H_2)-z(1-z)\widehat H_1\widehat H_2$. Then
\begin{equation}\label{eq:SP}
\begin{array}{c|cc}
&H&\widehat H\\\hline
S&H_1H_2&\Psi\\
P&\Psi&\widehat H_1\widehat H_2
\end{array}
\end{equation}
The initial value for a single-edge network is $(H,\widehat H)=(1,0)$.
\end{proposition}
\begin{proof}
A series composition is a one-vertex sum, so its broken-circuit complex is a join. Identifying its two outer terminals gives a parallel composition of the two networks. Formula~\eqref{eq:parallelH} applies to a parallel composition; after its terminals are identified, the result is a one-vertex sum of the two quotient graphs. The convention $\widehat H=0$ handles quotients with loops.
\end{proof}
This is a recursion for polynomial coefficients. Recursively generating the shelling itself requires additional records of facets, restriction faces, and boundary paths. Appendix~\ref{app:SP-shell} proves that such records close under composition for a specified lexicographic shelling of two-terminal series--parallel graphs.

\section{Orlik--Terao ideals and boundary composition of Gr\"obner bases}
Orient each edge and put $\ell_e=t_{\mathrm{head}(e)}-t_{\mathrm{tail}(e)}$. Define
\begin{equation}\label{eq:OT}
\begin{aligned}
I_G&=\ker\left(\QQ[y_e:e\in E(G)]\longrightarrow\QQ(t_v:v\in V(G)),\quad
y_e\longmapsto\ell_e^{-1}\right),\\
R_G&=\QQ[y]/I_G.
\end{aligned}
\end{equation}
Graphs in this section are loopless, but parallel edges are allowed. A circuit relation $\sum_{e\in C}a_e\ell_e=0$ gives the polynomial
\begin{equation}\label{eq:circuit-poly}
g_C=\sum_{e\in C}a_e\prod_{f\in C\setminus\{e\}}y_f.
\end{equation}
In the graphic case, the coefficients can be chosen in $\{1,-1\}$ according to the orientations.

\begin{theorem}[Proudfoot--Speyer \cite{ProudfootSpeyer}]\label{thm:PS}
The circuit polynomials $g_C$ form a universal Gr\"obner basis of $I_G$. For $e_1<\cdots<e_m$, use the lexicographic monomial order with $y_{e_1}\succ\cdots\succ y_{e_m}$. Then
\begin{equation}\label{eq:initial}
\inid_\prec I_G=I_{\BC(G)}
=\left\langle\prod_{e\in C\setminus\{\max C\}}y_e:C\text{ is a circuit}\right\rangle.
\end{equation}
\end{theorem}
Thus the standard monomials are precisely those whose supports are NBC faces. For a connected graph,
\begin{equation}\label{eq:hilb}
\Hilb(R_G,z)=\Hilb(\QQ[\BC(G)],z)=\frac{H_G(z)}{(1-z)^{|V|-1}}.
\end{equation}
The circuit-polynomial construction, the universal Gr\"obner basis theorem, and degeneration to the initial ring are established results, not new assertions of this paper.

\subsection{Path polynomials and gluing along two or three vertices}
Traverse a simple path $P$ from $a$ to $b$. Let $\varepsilon_{P,e}$ be its sign relative to the fixed orientation of $e$, and set
\begin{equation}\label{eq:path-poly}
m_P=\prod_{e\in P}y_e,\qquad
q_P=\sum_{e\in P}\varepsilon_{P,e}\prod_{f\in P\setminus\{e\}}y_f.
\end{equation}
After reciprocal substitution, $q_P=(t_b-t_a)m_P$. For concatenations of edge-disjoint paths, with compatible traversal directions,
\begin{equation}\label{eq:path-series}
m_{PQ}=m_Pm_Q,\qquad q_{PQ}=q_Pm_Q+m_Pq_Q.
\end{equation}

\begin{theorem}[Composition of local bases]\label{thm:GB-glue}
Suppose $G_1,G_2$ are edge-disjoint and $2\le |B|\le3$. Fix any global term order $\prec$, and let $\Gc_i$ be a Gr\"obner basis of $I_{G_i}$ for its restriction. For every pair of internally vertex-disjoint paths $P_i\subseteq G_i$ joining the same two boundary vertices, form
\begin{equation}\label{eq:cross-poly}
g_{P_1,P_2}=q_{P_1}m_{P_2}-m_{P_1}q_{P_2}.
\end{equation}
Adjoining all these polynomials to $\Gc_1\cup\Gc_2$ gives a Gr\"obner basis of $I_{G_1\cup G_2}$. Hypothesis~\eqref{eq:A} is unnecessary.
\end{theorem}
\begin{proof}
Along a simple circuit using both sides, the side containing the current edge changes a positive even number of times. Switching vertices lie in $B$ and are distinct. Since $|B|\le3$, there are exactly two switches. The circuit is therefore the union of two paths as above, and its circuit polynomial is \eqref{eq:cross-poly}. The leading monomial of a circuit contained in one side is divisible by a leading monomial from the corresponding local basis. By Theorem~\ref{thm:PS}, these leading monomials generate the entire initial ideal, proving the assertion.
\end{proof}

\subsection{Virtual complete graphs and elimination for arbitrary boundaries}
\begin{theorem}[Composition and elimination using a virtual clique]\label{thm:clique}
For an arbitrary boundary $B$, add a complete graph $Q$ on $B$ whose edges have new labels. Add a distinct virtual edge even when an original edge already joins the same pair, retaining the parallel edges. The virtual variables are shared by the two sides. Extending both local ideals to the polynomial ring on all edge variables, we have
\begin{align}
I_{G+Q}&=I_{G_1+Q}+I_{G_2+Q},\label{eq:clique-sum}\\
I_G&=\bigl(I_{G_1+Q}+I_{G_2+Q}\bigr)\cap\QQ[y_e:e\in E(G)].\label{eq:elimination}
\end{align}
Neither identity requires an edge-order hypothesis.
\end{theorem}
\begin{proof}
Local relations hold globally, so the right-hand side $J$ of \eqref{eq:clique-sum} is contained in the left-hand side. For the reverse inclusion, it suffices to consider circuit polynomials. Assign virtual edges to the first side, and suppose a circuit $C$ uses both sides. Decompose it into maximal monochromatic paths $P_1,\ldots,P_s$. Their switching vertices are distinct boundary vertices, and $s$ is a positive even integer.

Let $p_j$ be the virtual edge joining the endpoints of $P_j$. Define $q_{P_j}$ in the fixed direction of $p_j$, and let $\delta_j\in\{1,-1\}$ record whether the traversal of $C$ agrees with that direction. The local circuit relation between the path and its virtual edge gives
\[
m_{P_j}\equiv y_{p_j}q_{P_j}\pmod J.
\]
If $P_j$ is the virtual edge itself, the identity is immediate from $q_{P_j}=1$. Thus
\begin{align*}
g_C&=\sum_{j=1}^s\delta_jq_{P_j}\prod_{h\ne j}m_{P_h}\\
&\equiv\left(\prod_{j=1}^s q_{P_j}\right)
\left(\sum_{j=1}^s\delta_j\prod_{h\ne j}y_{p_h}\right)\pmod J.
\end{align*}
If $s=2$, then $p_1=p_2$ and $\delta_1=-\delta_2$, so the second factor vanishes. If $s\ge4$, the edges $p_1,\ldots,p_s$ form a simple circuit in $Q$, and the second factor is its circuit polynomial. Hence $g_C\in J$. Circuits contained in one side already belong to a local ideal. Theorem~\ref{thm:PS} now proves \eqref{eq:clique-sum}.

The kernel of reciprocal substitution restricted to the original edge variables is $I_G$. Eliminating the virtual variables from $I_{G+Q}$ therefore gives \eqref{eq:elimination}.
\end{proof}

\begin{corollary}
There is an algebra isomorphism
\[
R_{G+Q}\cong R_{G_1+Q}\otimes_{R_Q}R_{G_2+Q}.
\]
The union of local Gr\"obner bases need not itself be a Gr\"obner basis of the sum. A further computation in an elimination order may be necessary.
\end{corollary}
This is related to the established decomposition theory of Orlik--Terao algebras along modular flats \cite{Denham}; the proof above establishes \eqref{eq:clique-sum} directly through graphic circuits.

\subsection{The state filtration of the initial ring}
The shelling intervals give a decomposition of graded vector spaces
\begin{equation}\label{eq:Stanley}
\QQ[\BC(G)]=\bigoplus_T y^{R(T)}\QQ[y_e:e\in T].
\end{equation}
Indeed, the support of each standard monomial belongs to exactly one shelling interval. This is not a direct-product decomposition of rings.

Under \eqref{eq:A}, order the states as $\pi_1,\ldots,\pi_s$, and let $\Delta_j$ be the union of the faces in the first $j$ states. Set $\Delta_0$ equal to the void complex and $\Delta_s=\BC(G)$. Since finer states come first, each $\Delta_j$ is a subcomplex. In $R=\QQ[\BC(G)]$, put
\[
N_j=\ker(R\longrightarrow\QQ[\Delta_j]).
\]
Then $R=N_0\supseteq N_1\supseteq\cdots\supseteq N_s=0$, and
\begin{equation}\label{eq:filtration}
N_{j-1}/N_j\cong
\QQ[\Sc_{\pi_j}]\otimes_{\QQ}\QQ[K_{\pi_j}]
\end{equation}
as graded modules. The left-hand side is the relative Stanley--Reisner module $I_{\Delta_{j-1}}/I_{\Delta_j}$. Reading \eqref{eq:state-decomposition} in terms of monomial supports gives the isomorphism. Formula~\eqref{eq:global-R} further partitions the standard monomials in each quotient into intervals. Thus \eqref{eq:nonnegative} also decomposes the numerator in \eqref{eq:hilb}.

\section{Complexity: explicit bases versus coefficient computation}
\subsection{Composition cost in terms of boundary paths}
For a boundary of size two or three, let $p_i(a,b)$ be the number of simple $a$--$b$ paths in $G_i$. The number of candidate pairs is at most
\begin{equation}\label{eq:path-count}
M=\sum_{\{a,b\}\subseteq B}p_1(a,b)p_2(a,b).
\end{equation}
Given the local bases and the path lists, compatibility checking and writing the polynomials \eqref{eq:cross-poly} take $O(n^2M)$ edge-label operations, in addition to reading the lists: each polynomial has at most $n$ terms, and each term involves at most $n-1$ variables. This bound excludes the cost of generating the path lists themselves.

For a fixed edge order, let $\mathcal M$ be the inclusion-minimal broken circuits. The leading monomials of the reduced Gr\"obner basis are exactly $\{y^B:B\in\mathcal M\}$. Since $I_G$ is homogeneous,
\begin{equation}\label{eq:degree}
\max_{g\in\Gc_{\rm red}}\deg g=\max_{B\in\mathcal M}|B|
\end{equation}
when $I_G\ne0$. For a cycle $C_n$, the required degree is $n-1$.

For example, order the edges of $K_4$ by $13<24<12<23<34<14$. The broken circuit of the four-cycle $12341$ is $\{12,23,34\}$. Every broken circuit of a triangle contains $13$ or $24$, so this three-edge set is minimal, and the basis for this order requires a cubic polynomial. In contrast, an edge order arising from a suitable perfect elimination ordering gives a quadratic basis consisting of triangle polynomials. The existence of a favorable order must therefore be distinguished from the degree required by a prescribed order.

\begin{theorem}[Exponential output at treewidth two]\label{thm:exponential}
For each $k\ge2$, there is a simple graph $G_k$ with $3k+1$ vertices, $4k+1$ edges, and treewidth two, together with an edge order, such that the reduced Gr\"obner basis of its Orlik--Terao ideal has exactly $k+2^k$ elements and maximum degree $2k$. The order satisfies \eqref{eq:A} for a gluing along two vertices.
\end{theorem}
\begin{proof}
Take vertices $u=x_0,x_1,\ldots,x_k=v$ and two internal vertices $a_i,b_i$ for each interval. Let $G_1$ join $x_{i-1}$ to $x_i$ by the two paths $x_{i-1}a_ix_i$ and $x_{i-1}b_ix_i$ of length two. Let $G_2$ consist of the single edge $p=uv$, and place $p$ last in the order.

The circuits of $G_1$ are exactly the $k$ local four-cycles. A circuit containing $p$ is the union of $p$ with a $u$--$v$ path $P$ in $G_1$, and there are $2^k$ such paths. Since $p$ is largest, the associated broken circuit is $P$ itself, of cardinality $2k$. Each $P$ uses only two edges from any local four-cycle, so it contains no local three-edge broken circuit. Distinct paths have the same cardinality and do not contain one another. Since $k\ge2$, no local three-edge broken circuit contains one of these paths. Thus there are exactly $k+2^k$ minimal generators, and \eqref{eq:degree} gives the maximum degree.

A tree decomposition consists of the bags $\{u,x_{i-1},x_i\}$ in a chain, with leaf bags $\{x_{i-1},x_i,a_i\}$ and $\{x_{i-1},x_i,b_i\}$ attached to the corresponding chain bag. Every bag has at most three vertices. The graph contains a cycle, so its treewidth is exactly two.
\end{proof}

\begin{figure}[htbp]\centering
\begin{tikzpicture}[scale=1,vertex/.style={circle,fill=black,inner sep=1.7pt}]
\foreach \i in {0,1,2,3}{\node[vertex] (x\i) at (2.4*\i,0) {};\node[below=4pt] at (x\i) {$x_{\i}$};}
\foreach \i/\j in {1/0,2/1,3/2}{
  \node[vertex] (a\i) at (2.4*\i-1.2,.9) {};
  \node[vertex] (b\i) at (2.4*\i-1.2,-.9) {};
  \node[above=3pt] at (a\i) {$a_{\i}$};
  \node[below=3pt] at (b\i) {$b_{\i}$};
  \draw (x\j)--(a\i)--(x\i)--(b\i)--(x\j);
}
\draw[thick] (x0) .. controls (0,2.8) and (7.2,2.8) .. node[pos=.5,above] {$p=uv\quad(\text{largest edge})$} (x3);
\node[left=5pt] at (x0) {$u$};\node[right=5pt] at (x3) {$v$};
\end{tikzpicture}
\caption{The graph of Theorem~\ref{thm:exponential} for $k=3$. Independent choices in the three intervals give eight boundary paths.}
\end{figure}
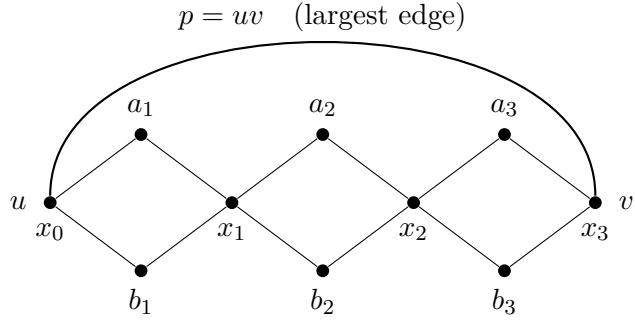

Theorem~\ref{thm:exponential} is a lower bound for explicit basis output. It does not give an exponential lower bound for a compressed representation or for computing the Hilbert series. For example, the polynomials $H_G$ of this family can be computed using \eqref{eq:SP}.

\subsection{State polynomials at bounded treewidth}
\begin{theorem}[Coefficient computation using partition states]\label{thm:DP}
Suppose a tree decomposition of width $w$ has a bag containing $B$, and let $N$ be the number of bags after conversion to a nice tree decomposition, including edge-introduction nodes. Put $b=\Bell(w+1)$ and $n=|V(G)|$. All state polynomials $N_{G,\pi}(q,-1,1)$, and hence all $J_{G,\pi}(z)$, can be computed in
\begin{equation}\label{eq:DPcost}
O(Nb^2n^2)
\end{equation}
integer arithmetic operations using naive polynomial arithmetic. For fixed width and a standard polynomial-size decomposition, this gives polynomial time in the input size.
\end{theorem}
\begin{proof}
For every partition $\rho$ of a bag $X$, sum the weights $(-1)^{|A|}q^{c_X(A)}$ of edge sets in the processed subgraph that induce $\rho$ on $X$. Here $c_X(A)$ counts components that no longer meet the bag. The transitions are as follows.
\begin{enumerate}
\item When a vertex is introduced, add it as a singleton block.
\item When an edge is introduced, add the contribution omitting it unchanged; for the contribution using it, merge the endpoint blocks and reverse the sign.
\item When a vertex is forgotten, multiply by $q$ if it is the last vertex of its block in the bag; otherwise simply delete it from the partition.
\item At a join node, send a pair $\rho_1,\rho_2$ to $\rho_1\vee\rho_2$ and multiply their polynomials. Components disjoint from the bag are counted additively.
\end{enumerate}
Root the decomposition at a bag containing $B$ and forget its other vertices above it, so that the final bag is $B$. The exponent at the root is then $\kappa(A)-|\pi|$, giving $N_{G,\pi}$. Degrees are at most $n$; each join has at most $b^2$ state pairs, and naive multiplication takes $O(n^2)$ operations. This proves \eqref{eq:DPcost}. Apply the change of variables \eqref{eq:J-N} at the end. Coefficients are signed counts of edge subsets and sums involving binomial coefficients, so their bit lengths are also polynomially bounded in $n+|E|$.
\end{proof}
Evaluation and computation of the Tutte polynomial at bounded treewidth have been studied previously \cite{Noble}. The point here is to extract boundary-state coefficients and distinguish their computation from explicit basis output. We do not claim that \eqref{eq:DPcost} is an optimal bound.

\section{Convex hulls of NBC spanning trees and lifts of states}
\subsection{Base polytopes and broken-circuit complexes}
The graphic base polytope is
\[
P_G=\conv\{\one_T:T\text{ is a spanning tree}\}\subseteq\RR^{E(G)}.
\]
Its edges correspond to single-edge exchanges \cite{GGMS}. Write $Q_{\NBC}(G)$ for the convex hull obtained by retaining only NBC spanning trees; in general, it differs from $P_G$. The complex $\BC(G)$ is an abstract simplicial complex whose vertices are graph edges, whereas the vertices of $Q_{\NBC}(G)$ are indicator vectors of spanning trees.

The largest edge belongs to every NBC spanning tree, so $\BC(G)$ is a cone with that edge as apex. It therefore cannot generally be identified with the boundary complex of a convex polytope. For a triangle, for example, the broken-circuit complex consists of two line segments meeting at a common endpoint and has $H=1+z$, whereas its NBC convex hull is a primitive lattice segment with Ehrhart $h^*$-polynomial $1$. These polynomials do not agree in general.

Define the state of a global NBC spanning tree $T$ to be $\pi(T\cap E(G_1))$, and put
\[
Q_\pi=\conv\{\one_T:T\text{ is an NBC spanning tree of state }\pi\}.
\]
Every component of $T\cap E(G_1)$ meets the boundary, so
\begin{equation}\label{eq:coarseheight}
|T\cap E(G_1)|=n_1-|\pi|.
\end{equation}
Thus the linear function $x(E(G_1))$ distinguishes the number of blocks, but need not distinguish states with the same number of blocks.

\begin{proposition}[An obstruction to separating states by a linear height]\label{prop:midpoint}
Even for a three-vertex gluing satisfying \eqref{eq:A}, convex hulls of distinct states may intersect. In that case a single linear height in the original edge coordinates cannot put the states in consecutive strictly separated blocks.
\end{proposition}
\begin{proof}
Take $K_{3,3}$ with parts $\{u,v,w\}$ and $\{a,b,c\}$, and label its edges by
\[
\begin{array}{c|ccccccccc}
\text{label}&1&2&3&4&5&6&7&8&9\\\hline
\text{edge}&ua&ub&va&vb&wa&wb&uc&vc&wc.
\end{array}
\]
Let $G_1$ contain edges $1,\ldots,6$, let $G_2$ contain edges $7,8,9$, and use the label order. The following four edge sets are NBC spanning trees:
\[
T_1=13689,\quad T_2=24589,\qquad S_1=14589,\quad S_2=23689,
\]
where a string of digits denotes the corresponding edge set. Each is a spanning path. Adding any omitted edge creates a fundamental circuit containing a larger edge, either $8$ or $9$, so the tree is NBC. The states of $T_1,T_2$ are $uv|w$, and those of $S_1,S_2$ are $uw|v$. Moreover,
\begin{equation}\label{eq:midpoint}
\one_{T_1}+\one_{T_2}=\one_{S_1}+\one_{S_2}.
\end{equation}
The two state polytopes therefore share a midpoint. If every vertex in one state had strictly greater height than every vertex in the other, the sums of the heights on the two sides of \eqref{eq:midpoint} would differ, a contradiction. Breaking ties by state information uses data beyond a linear height in the original edge coordinates.
\end{proof}

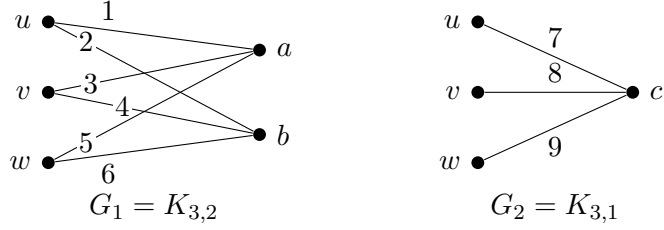
\begin{figure}[htbp]\centering
\begin{tikzpicture}[vertex/.style={circle,fill=black,inner sep=1.7pt},scale=.93]
\foreach \v/\y in {u/2,v/1,w/0}{\node[vertex,label=left:{$\v$}] (\v) at (0,\y) {};}
\foreach \v/\y in {a/1.6,b/.4}{\node[vertex,label=right:{$\v$}] (\v) at (3,\y) {};}
\draw (u)--node[above,pos=.27] {$1$}(a);
\draw (u)--node[fill=white,inner sep=.8pt,pos=.16] {$2$}(b);
\draw (v)--node[fill=white,inner sep=.8pt,pos=.18] {$3$}(a);
\draw (v)--node[fill=white,inner sep=.8pt,pos=.34] {$4$}(b);
\draw (w)--node[fill=white,inner sep=.8pt,pos=.16] {$5$}(a);
\draw (w)--node[below,pos=.27] {$6$}(b);
\node at (1.5,-.65) {$G_1=K_{3,2}$};
\begin{scope}[xshift=6.1cm]
\foreach \v/\y in {u/2,v/1,w/0}{\node[vertex,label=left:{$\v$}] (q\v) at (0,\y) {};}
\node[vertex,label=right:{$c$}] (qc) at (2.2,1) {};
\draw (qu)--node[above] {$7$}(qc);\draw (qv)--node[above] {$8$}(qc);\draw (qw)--node[below] {$9$}(qc);
\node at (1.1,-.65) {$G_2=K_{3,1}$};
\end{scope}
\end{tikzpicture}
\caption{An example with intersecting state polytopes. Identifying the corresponding vertices $u,v,w$ on the two sides gives $K_{3,3}$.}
\end{figure}

\subsection{Cayley lifts and realization by a linear order}
Let $\mathcal P$ be the set of nonempty states, and let $e_\pi$ be the corresponding unit vectors. The polytope
\begin{equation}\label{eq:Cayley}
\mathcal C=\conv\bigcup_{\pi\in\mathcal P}(Q_\pi\times e_\pi)
\end{equation}
is the usual Cayley construction \cite{Huber}.
\begin{proposition}\label{prop:Cayley}
Each $Q_\pi\times e_\pi$ is an exposed face of $\mathcal C$. For $\lambda_\pi\ge0$ with $\sum\lambda_\pi=1$,
\[
\{x:(x,\lambda)\in\mathcal C\}=\sum_\pi\lambda_\pi Q_\pi.
\]
Under \eqref{eq:A}, let $A_\pi$ be the convex hull of indicator vectors of the local maximal forests of state $\pi$, and let $B_\pi$ be the convex hull of the facets of $K_\pi$. Then
\begin{equation}\label{eq:polyproduct}
Q_\pi=A_\pi\times B_\pi,\qquad
\#\operatorname{Vert}(Q_\pi)=J_{G_1,\pi}(1)H_{G_2/\pi}(1).
\end{equation}
\end{proposition}
\begin{proof}
Maximizing the state coordinate $y_\pi$ exposes $Q_\pi\times e_\pi$. Grouping a convex combination by state gives the fiber identity. The product description of facets in Theorem~\ref{thm:block} gives \eqref{eq:polyproduct}. Every vector in a set of distinct $0$--$1$ vectors is a vertex of their convex hull, so the number of vertices is the product of the facet counts.
\end{proof}

\begin{proposition}[Block orders from a lifted linear height]\label{prop:liftheight}
Under \eqref{eq:A}, choose the local shellings in Theorem~\ref{thm:block} from lexicographic shellings. The resulting global block-shelling order is realized by a single linear height on the vertices of $\mathcal C$. For realizing the order alone, a lift by one integer state label also suffices.
\end{proposition}
\begin{proof}
List the states as $\pi_0,\ldots,\pi_{s-1}$. Write $E(G_1)=\{e_1<\cdots<e_{m_1}\}$ and $E(G_2)=\{f_1<\cdots<f_{m_2}\}$, and set
\[
\psi(x)=2^{m_2}\sum_{i=1}^{m_1}2^{i-1}x_{e_i}
+\sum_{j=1}^{m_2}2^{j-1}x_{f_j},\qquad M=2^{m_1+m_2}.
\]
Then $0\le\psi(\one_T)<M$, and decreasing $\psi$ gives the lexicographic order with $F$ outermost and $U$ innermost. Thus decreasing
\[
L(x,y)=M\sum_{j=0}^{s-1}(s-1-j)y_{\pi_j}+\psi(x)
\]
gives the required state-block order. Its shelling property follows from the interval proof of \eqref{eq:global-R}, not from linearity of the height alone. For the lift $(\one_T,j)$, use $-Mj+\psi(x)$.
\end{proof}

Two additional coordinates $(j,j^2)$ suffice to make each state an exposed face. Indeed, maximizing $2j\xi-\eta$ exposes state $j$, since its value at state $h$ is $j^2-(h-j)^2$. This asserts exposure of the states, not equality with the face lattice of the Cayley polytope.

\begin{example}[An edge of the lifted polytope that is not a single exchange]
In Proposition~\ref{prop:midpoint}, the NBC spanning trees of state $uv|w$ are
\[
12489,\quad13489,\quad13689,\quad23489,\quad24589.
\]
The candidate forests on the $G_1$ side are $123,124,134,234$ when $w$ is isolated, and $136,245$ when its component also contains an interior vertex. The set $123$ is removed because it is the broken circuit of the circuit $1234$. The only NBC basis of $G_2/\pi$ is $89$.

The linear function $3y_{uv|w}+x_5+x_6$ on $\mathcal C$ attains its maximum value $4$ precisely at $13689$ and $24589$. On all other states its value is at most $2$. Hence the segment joining these two vertices is an edge of $\mathcal C$, although the endpoints differ by three edges on each side. Its direction is not a single-exchange direction $e_i-e_j$. Thus $\mathcal C$ in these coordinates is not, in general, a matroid base polytope.
\end{example}

\subsection{Toric ideals of the lifted configuration}
Using variables $p_T$ indexed by NBC spanning trees, define
\[
I_{\mathcal A}=\ker(p_T\mapsto x^T),\qquad
I_{\mathcal C}=\ker(p_T\mapsto s_{\pi(T)}x^T).
\]
These are different from the Orlik--Terao ideal in edge variables. Equation~\eqref{eq:midpoint} gives
\[
p_{T_1}p_{T_2}-p_{S_1}p_{S_2}\in I_{\mathcal A},
\]
but this binomial is not in $I_{\mathcal C}$, since its two terms have different state degrees. The lift adds a state grading.

If $A$ is the configuration matrix and a weight comes from a linear height on its points, then $w=A^{\mathsf T}c$. Every toric relation $A\alpha=A\beta$ satisfies
\[
w\cdot\alpha=w\cdot\beta.
\]
Thus a linear height on the configuration alone cannot choose an initial term of such a binomial. Realizing a block order on the lifted configuration is different from constructing a squarefree toric initial ideal or a regular unimodular triangulation.

\section{Scope and further questions}
The constructions above establish the following assertions. NBC faces with a prescribed boundary state admit relative shellings obtained from complexes separating chosen representatives. When every edge of one side precedes every edge of the other, a one-step block composition is possible for an arbitrary boundary size. Its correction to restriction faces is a union of paths within boundary blocks, and its nonnegative generating-function decomposition agrees with the chromatic specialization of the Negami splitting identity.

The algebraic compositions do not require this edge-order hypothesis. For boundaries of size two or three, local Gr\"obner bases can be supplemented by circuit polynomials formed from paths on opposite sides. For an arbitrary boundary, adding a virtual clique and eliminating its variables composes the ideals. This algebraic closure does not guarantee computation with a small record: explicit bases can have exponential size at treewidth two, whereas state-polynomial coefficients can be computed in polynomial time at fixed treewidth.

The remaining questions can be separated as follows.
\begin{enumerate}
\item \textbf{Prescribed interleaved edge orders.} Augment connectivity partitions with order information, such as largest path edges, and find sufficient conditions for attaching states as consecutive blocks. NBC recognition and the block-shelling property require separate proofs.
\item \textbf{Closure of shelling records along tree decompositions.} Beyond the two-terminal series--parallel setting of Appendix~\ref{app:SP-shell}, find rules that update facets, restriction faces, and boundary paths of auxiliary complexes using only child records at introduction, forgetting, and join steps, and establish output-sensitive complexity bounds. The dynamic program for coefficients does not resolve this question.
\item \textbf{Toric realizations.} Determine to what extent initial ideals obtained from suitable non-affine vertex weights on the state-labeled Cayley configuration correspond to the relative state decomposition of the broken-circuit complex.
\item \textbf{A three-variable extension.} Find conditions under which the present argument, based on sign reversal at the chromatic specialization and NBC forests, extends to a nonnegative structure retaining all three variables.
\end{enumerate}
These questions are not included among the conclusions of the main theorems. The established correspondence concerns relative shellings, splitting identities, and initial rings under the stated hypotheses.

\clearpage
\appendix
\section{Notation and rank conventions}
\begin{longtable}{p{.23\textwidth}p{.70\textwidth}}
\toprule Notation & Meaning\\\midrule\endhead
$\BC(G)$ & The complex of edge sets containing no broken circuit, where the largest edge is deleted from each circuit.\\
$\pi(A)$ & The connectivity partition induced by $A$ on the boundary $B$.\\
$\Sc_\pi$ & The relative complex of NBC faces with state exactly $\pi$.\\
$D_S$ & The subcomplex of NBC faces separating all vertices of the representative set $S$.\\
$P_\pi(F)$ & The union of paths in $F$ from each block representative to the other vertices of that boundary block.\\
$K_\pi$ & The complex $\BC(G_2/\pi)$; it is void if the quotient has a loop.\\
$J_{G,\pi}$ & The relative $h$-polynomial of $\Sc_\pi$, using rank $|V(G)|-|\pi|$.\\
$H_G$ & The $h$-polynomial of $\BC(G)$ for connected $G$, using rank $|V(G)|-1$.\\
$\widehat H_G$ & The polynomial $H_{G/uv}$ after identification of two terminals; it is zero if a loop is created.\\
$N_{G,\pi}$ & The state sum in which the exponent of $t$ counts components avoiding the boundary.\\
$I_G$ & The Orlik--Terao ideal defined by reciprocals of the edge linear forms.\\
$Q_\pi$ & The convex hull of indicator vectors of NBC spanning trees of state $\pi$.\\
$\mathcal C$ & The Cayley polytope obtained by adding a state unit vector to each $Q_\pi$.\\\bottomrule
\end{longtable}
The relevant ranks and substitution are
\[
\begin{aligned}
r(G)&=n_1+n_2-k-1,\\
r(\Sc_\pi)&=n_1-|\pi|,\\
r(K_\pi)&=n_2-k+|\pi|-1,\\
q&=(z-1)/z.
\end{aligned}
\]
A relative $h$-polynomial uses the rank of its relative complex. The powers of $(1-z)$ introduced by changing to a common ambient rank must be checked separately from the operation of adding paths to restriction faces.

\section{Two-terminal recognition under interleaved orders and recursive series--parallel shellings}\label{app:SP-shell}
This appendix describes how far the two-terminal recursion closes when additional edge-order information is recorded. Label the edges by $1,\ldots,m$, and use $0$ as an auxiliary value smaller than every edge. All shellings in this appendix use the lexicographic order of Proposition~\ref{prop:known-shell}.

\subsection{NBC recognition using largest-edge records}
For an NBC face $A$ of a graph $G$ with terminals $s,t$, if $A$ connects the terminals, let $P_A$ be their unique path and put $\alpha(A)=\max P_A$. If $A$ separates the terminals, put
\begin{equation}\label{eq:beta}
\begin{split}
\beta(A)=\max\bigl(\{0\}\cup\{e\in E(G)\setminus A:
&\ A\cup\{e\}\text{ connects }s,t,\\[-2pt]
&\ e\text{ is the largest edge of that path}\}\bigr).
\end{split}
\end{equation}

\begin{proposition}\label{prop:alpha-beta}
Suppose $G_1,G_2$ are edge-disjoint and share only the terminals $s,t$. The union of local NBC faces $A_i$ is NBC exactly in the following cases:
\begin{enumerate}
\item If both sides separate the terminals, the union is always NBC.
\item If both sides connect the terminals, the union is never NBC.
\item If $A_i$ connects the terminals and $A_j$ separates them, the condition is $\alpha(A_i)>\beta(A_j)$.
\end{enumerate}
\end{proposition}
\begin{proof}
When both sides connect the terminals, the union contains a circuit. When both separate them, the union cannot contain the broken circuit of a circuit using both sides. If exactly one side connects them, any new broken circuit in the union consists of the entire terminal path on that side and a terminal path missing one edge on the separating side. The missing edge $e$ is largest in the whole circuit precisely when it is a candidate in \eqref{eq:beta} and $e>\alpha(A_i)$. Since the edge sets are disjoint, $\alpha=\beta>0$ is impossible. The criterion follows.
\end{proof}

\subsection{Local computation of restriction faces}
A facet $F$ of the terminal-separating complex $D$ is a spanning forest with two components. Define
\begin{equation}\label{eq:gamma}
\gamma(F)=\max\{e\in E(G):\text{the endpoints of $e$ lie in distinct components of $F$}\}.
\end{equation}
This set is nonempty because $G$ is connected. Let $T_i$ be an NBC spanning tree of $G_i$, and let $F_j$ be a facet of $D_j$.

\begin{proposition}\label{prop:restriction-mixed}
If $\alpha(T_i)>\beta(F_j)$, then $T=T_i\cup F_j$ is a global NBC spanning tree, and its restriction face in the lexicographic shelling is
\begin{equation}\label{eq:R-mixed}
R_G(T)=R_{G_i}(T_i)\cup R_{D_j}(F_j)
\cup\{b\in P_{T_i}:b<\gamma(F_j)\}.
\end{equation}
The facets of $K=\BC(G/st)$ are precisely the facets $F$ of $D$ with $\beta(F)=0$. Their lexicographic restriction faces equal $R_D(F)$.
\end{proposition}
\begin{proof}
In the lexicographic NBC shelling, an edge $b$ of a tree $T$ belongs to its restriction face if and only if an external edge larger than $b$ crosses the cut obtained by deleting $b$. Exchange with the largest edge of a fundamental cut preserves the NBC property \cite{Ziegler}. For $D$, the same exchange criterion applies by regarding it as the link of a largest virtual edge.

If $b\in F_j$ is deleted, one of the resulting smaller components avoids the boundary. Global exchanges therefore coincide with exchanges in the local representative-separating forest. The same argument applies to $b\in T_i\setminus P_{T_i}$. If $b\in P_{T_i}$, deleting it separates the terminals. In addition to exchange candidates on the first side, all edges crossing the two components of $F_j$ become candidates. Their largest value is $\gamma(F_j)$, proving \eqref{eq:R-mixed}.

After terminal identification, a new violation of the NBC criterion can occur exactly when an edge joining the two terminal components is a candidate in \eqref{eq:beta}. Thus a facet of $D$ is a facet of $K$ exactly when $\beta=0$. Its fundamental cuts after identification agree with the cuts obtained by adding a virtual edge to $F$ and deleting a real edge. The restriction faces consequently agree as well.
\end{proof}

\subsection{Record updates for series--parallel networks}
For every NBC spanning tree, record $(T,R,P,\alpha)$; for every maximal terminal-separating spanning forest, record $(F,R,\beta,\gamma)$. For a single edge $e$, the initial records are
\[
(\{e\},\varnothing,\{e\},e),\qquad
(\varnothing,\varnothing,e,e).
\]
Write $\mathsf T$ and $\mathsf F$ for the collections of tree and forest records.

\begin{theorem}[A closed record recursion for two-terminal series--parallel graphs]\label{thm:SP-record}
The following updates recursively generate all NBC spanning trees, all maximal terminal-separating forests, and their lexicographic restriction faces for a two-terminal series--parallel graph with any edge order.

For a series composition, form trees from $\mathsf T_1\times\mathsf T_2$ by
\[
T=T_1\cup T_2,\quad R=R_1\cup R_2,\quad
P=P_1\cup P_2,\quad \alpha=\max(\alpha_1,\alpha_2).
\]
Form forests from $\mathsf T_i\times\mathsf F_j$ for $\{i,j\}=\{1,2\}$. For records $(T,R_T,P,\alpha)$ and $(F,R_F,\beta,\gamma)$, set
\begin{equation}\label{eq:SP-series-record}
\begin{split}
F'&=T\cup F,\qquad
R'=R_T\cup R_F\cup\{b\in P:b<\gamma\},\\
\beta'&=\begin{cases}\beta,&\beta>\alpha,\\0,&\beta\le\alpha,\end{cases}
\qquad\gamma'=\gamma.
\end{split}
\end{equation}

For a parallel composition, form forests from $\mathsf F_1\times\mathsf F_2$ by
\[
F=F_1\cup F_2,\quad R=R_1\cup R_2,\quad
\beta=\max(\beta_1,\beta_2),\quad\gamma=\max(\gamma_1,\gamma_2).
\]
For trees, retain only the pairs in $\mathsf T_i\times\mathsf F_j$ with $\alpha>\beta$, and set
\begin{equation}\label{eq:SP-par-record}
\begin{aligned}
T'&=T\cup F,& R'&=R_T\cup R_F\cup\{b\in P:b<\gamma\},\\
P'&=P,&\alpha'&=\alpha.
\end{aligned}
\end{equation}
Sorting each resulting collection in the global lexicographic order gives a shelling of the corresponding complex.
\end{theorem}
\begin{proof}
A spanning tree of a series composition is a pair of local spanning trees. A terminal-separating spanning two-forest consists of a spanning tree on one side and a terminal-separating two-forest on the other. The global terminal path is the concatenation of local terminal paths. Only edges on the forest side cross the forest cut, so $\gamma'=\gamma$. A candidate edge for $\beta$ is largest on the completed global terminal path if and only if it also exceeds the largest edge $\alpha$ of the tree-side terminal path. This gives the formula for $\beta'$ in \eqref{eq:SP-series-record}.

For a series-composition forest, test the restriction face in the tree obtained by adjoining a largest virtual edge. Deleting an edge on the tree-side terminal path exposes exchange candidates across the forest-side terminal cut. Away from the terminal path, the exchange test is local. This proves the formula for $R'$. For a series-composition tree, there are no additional exchange candidates, so its restriction face is simply the union of the local restriction faces.

A terminal-separating two-forest in a parallel composition is a pair of local terminal-separating two-forests. Each completing path and each cut edge belongs to one side, so $\beta$ and $\gamma$ are the maxima of the local values. Deleting an edge of either forest separates off a component avoiding the boundary, and exchanges into this component occur only on the same side. Hence the restriction face is the union of the local ones. For spanning trees, Propositions~\ref{prop:alpha-beta} and \ref{prop:restriction-mixed} give \eqref{eq:SP-par-record}. These classifications exhaust the spanning trees and spanning two-forests without repetition, since the edge sets of the sides are disjoint. The computed restriction faces agree with the lexicographic ones, proving the final assertion.
\end{proof}

\begin{corollary}[Output-sensitive polynomial bound]\label{cor:SP-output}
Given a series--parallel decomposition tree, all records in Theorem~\ref{thm:SP-record} can be generated explicitly in time polynomial in the sum of the input size and the number of final output records. This does not assert polynomial time in the input size alone.
\end{corollary}
\begin{proof}
Let $t_i,f_i$ denote the numbers of tree and forest records of a subnetwork. Both are at least one. For either composition, the parent's forest count is at least the forest count of each child. Moreover, deleting an edge from the terminal path of an NBC spanning tree gives a terminal-separating NBC forest. If the subnetwork has $m_i$ edges, each such forest has at most $m_i$ possible one-edge extensions. Thus $t_i\le m_i f_i$.

Let $f$ be the forest count at the root and $m$ the total number of edges. Every node has at most $(m+1)f$ records. At each node, inspect pairs of child records and update edge sets, paths, and largest-edge data in polynomial time per pair. Including the number of decomposition nodes in the input gives a polynomial bound in the final record count and the input size.
\end{proof}

This recursion concerns the records of the specified lexicographic shelling. It does not assert preservation of arbitrary local shellings or consecutive blocks indexed by $(\alpha,\beta)$. Nor does it imply closure of the same finite collection of record types for general tree decompositions with larger boundaries.


\begin{thebibliography}{99}
\small\raggedright
\bibitem{Negami}
S. Negami, Polynomial invariants of graphs,
\emph{Trans. Amer. Math. Soc.} \textbf{299} (1987), no.~2, 601--622.
\href{https://doi.org/10.1090/S0002-9947-1987-0869224-1}{doi:10.1090/S0002-9947-1987-0869224-1}.

\bibitem{NegamiKawagoe}
S. Negami and K. Kawagoe, Polynomial invariants of graphs with state models,
\emph{Discrete Appl. Math.} \textbf{56} (1995), no.~2--3, 323--331.
\href{https://doi.org/10.1016/0166-218X(94)00093-S}{doi:10.1016/0166-218X(94)00093-S}.

\bibitem{Bjorner}
A. Bj\"orner, The homology and shellability of matroids and geometric lattices,
in N. White (ed.), \emph{Matroid Applications}, Encyclopedia Math. Appl.~40,
Cambridge University Press, Cambridge, 1992, 226--283.

\bibitem{Ziegler}
G. M. Ziegler, Matroid shellability, $\beta$-systems, and affine hyperplane arrangements,
\emph{J. Algebraic Combin.} \textbf{1} (1992), 283--300.
\href{https://www.mi.fu-berlin.de/math/groups/discgeom/ziegler/Preprintfiles/030PREPRINT.pdf}{Author's preprint}.

\bibitem{ProudfootSpeyer}
N. Proudfoot and D. Speyer, A broken circuit ring,
\emph{Beitr. Algebra Geom.} \textbf{47} (2006), no.~1, 161--166.
\href{https://arxiv.org/abs/math/0410069}{arXiv:math/0410069}.

\bibitem{Denham}
G. Denham, M. Garrousian, and \c{S}. O. Toh\u{a}neanu,
Modular decomposition of the Orlik--Terao algebra,
\emph{Ann. Comb.} \textbf{18} (2014), no.~2, 289--312.
\href{https://doi.org/10.1007/s00026-014-0223-z}{doi:10.1007/s00026-014-0223-z}.

\bibitem{Burgos}
J. M. Burgos, Singularities in Negami's splitting formula for the Tutte polynomial,
\href{https://arxiv.org/abs/1605.05499v3}{arXiv:1605.05499v3}, 2017.

\bibitem{Noble}
S. D. Noble, Evaluating the Tutte polynomial for graphs of bounded tree-width,
\emph{Combin. Probab. Comput.} \textbf{7} (1998), no.~3, 307--321.

\bibitem{GGMS}
I. M. Gel'fand, R. M. Goresky, R. D. MacPherson, and V. V. Serganova,
Combinatorial geometries, convex polyhedra, and Schubert cells,
\emph{Adv. Math.} \textbf{63} (1987), 301--316.
\href{https://doi.org/10.1016/0001-8708(87)90059-4}{doi:10.1016/0001-8708(87)90059-4}.

\bibitem{Huber}
B. Huber, J. Rambau, and F. Santos,
The Cayley Trick, lifting subdivisions and the Bohne--Dress theorem on zonotopal tilings,
\emph{J. Eur. Math. Soc.} \textbf{2} (2000), no.~2, 179--198.
\href{https://ems.press/journals/jems/articles/113}{Publisher's article page}.
\end{thebibliography}
\end{document}